\documentclass[10pt,reqno]{amsart}
\usepackage{hyperref}\hypersetup{colorlinks=true, citecolor=blue}
\usepackage{graphicx}
\usepackage{xfrac}
\usepackage{amsfonts,amsmath,amssymb,epsfig,mathrsfs,bm}
\usepackage{ifthen,color,mathtools}
\usepackage{epic,eepic}
\usepackage{esint,yfonts,stmaryrd}
\usepackage{a4wide,accents}

\newtheorem{theorem}{Theorem}[section]
\newtheorem{lemma}[theorem]{Lemma}
\newtheorem{proposition}[theorem]{Proposition}
\newtheorem{corollary}[theorem]{Corollary}

\theoremstyle{definition}

\theoremstyle{remark}
\newtheorem{remark}[theorem]{Remark}
\theoremstyle{remark}

\numberwithin{equation}{section}
\newcommand{\R}{\mathbb{R}}

\newcommand{\diw}{\mathrm{div}\,}
\newcommand{\Ln}{\mathcal{L}^n}

\newcommand{\LL}{\mathbb{L}}

\newcommand\Sing{\textup{Sing}}
\newcommand\Reg{\textup{Reg}}

\title[Weiss' and Monneau's quasi-monotonicity formulas]{Quasi-monotonicity formulas for classical obstacle problems with Sobolev coefficients and applications}

\author[M.~Focardi]{Matteo Focardi}
\address{DiMaI, Universit\`a degli Studi di Firenze}
\curraddr{Viale Morgagni 67/A, 50134 Firenze (Italy)}
\email{matteo.focardi@unifi.it}
\author[F.~Geraci]{Francesco Geraci}
\address{DiMaI, Universit\`a degli Studi di Firenze}
\curraddr{Viale Morgagni 67/A, 50134 Firenze (Italy)}
\email{geraci@math.unifi.it}
\author[E.~Spadaro]{Emanuele Spadaro}
\address{Universit\`a di Roma ``La Sapienza''}
\curraddr{P.le Aldo Moro 5,  Roma (Italy)}
\email{spadaro@mat.uniroma1.it}
\thanks{
E.~S.~has been supported by the ERC-STG Grant n. 759229
HiCoS ``Higher Co-dimension Singularities: Minimal Surfaces and 
the Thin Obstacle Problem''.
M.~F., F.~G. and E.~S. are members of the Gruppo Nazionale per
l'Analisi Matematica, la Probabilit\`a e le loro Applicazioni (GNAMPA)
of the Istituto Nazionale di Alta Matematica (INdAM). }



\date{}
\begin{document}
\begin{abstract}
We establish Weiss' and Monneau's type quasi-monotonicity formulas for quadratic energies having matrix of coefficients
in a Sobolev space $W^{1,p}$, $p>n$, and provide an application to the corresponding free boundary analysis for the 
related classical obstacle problems.
\end{abstract}

\maketitle

%%%%%%%%%%%%%%%%%%%%%%%%%%%%%%%%%%%
%
%	INTRODUCTION
%
%%%%%%%%%%%%%%%%%%%%%%%%%%%%%%%%%%%
\section{Introduction}\label{s:intro}

The aim of this short note is to extend the range of validity of Weiss' and Monneau's type quasi-monotonicity 
formulas to classical obstacle problems involving quadratic forms having matrix of coefficients in a Sobolev 
space $W^{1,p}$, with $p>n$. Such results are instrumental to pursue the variational approach for the analysis
of the corresponding free boundaries in classical obstacle problems.
More precisely, we consider the functional $\mathscr{E}:W^{1,2}(\Omega)\to\R$ given by
 \begin{equation}\label{e:enrg}
 \mathcal{E}(v):=\int_\Omega \big(\langle \mathbb{A}(x)\nabla v(x),\nabla v(x)\rangle +2h(x)v(x)\big)\,dx,
\end{equation}
and study regularity issues related to its unique minimizer $w$ on the set %${\mathbb{K}_{\psi,g}}$ (see also \eqref{e:Kpsi})
\[
\mathcal{K}_{\psi,g} := \big\{v\in W^{1,2}(\Omega):\,v\geq \psi\;\;\Ln\,\text{ a.e. on } \Omega,\,
\textup{Tr}(v)=g \,\text{ on } \partial\Omega\big\}\,.
\]
Here $\Omega\subset \R^n$ is a bounded open set, $n\geq 2$, $\psi\in C^{1,1}_{loc}(\Omega)$ and $g\in H^{\sfrac12}(\partial\Omega)$,
are such that $\psi\leq g$ $\mathcal{H}^{n-1}$-a.e on $\partial\Omega$,
$\mathbb{A}:\Omega\to \R^{n\times n}$ is a matrix-valued field and $f:\Omega\to \R$ is a function satisfying:
\begin{itemize}
\item[(H1)] $\mathbb{A}\in W^{1,p}(\Omega;\R^{n\times n})$ with $p>n$;  
\item[(H2)] $\mathbb{A}(x)=\left(a_{ij}(x)\right)_{i,j=1,\dots,n}$ symmetric, continuous and coercive, 
  that is $a_{ij}=a_{ji}$ in $\Omega$ for all $i,\,j\in\{1,\ldots,n\}$,
  and for some $\Lambda\geq 1$ 
  \begin{equation}\label{A coerc cont}
   \Lambda^{-1}|\xi|^2\leq \langle \mathbb{A}(x)\xi, \xi\rangle \leq \Lambda|\xi|^2
%   \qquad\qquad \Ln \,\,\textrm{a.e.}\,\,\Omega, \,\,\forall \xi\in \R^n;
  \end{equation}
 for all $x\in\Omega$, $\xi\in \R^n$; 
 \item[(H3)] $f:=h-\diw(\mathbb{A}\nabla\psi)>c_0$ $\Ln$ a.e. $\Omega$, for some $c_0>0$, and
 $f$ is Dini-continuous, namely
  \begin{equation}\label{e:Dini continuity}
  \int_0^1 \frac{\omega_f(t)}{t}\,dt < \infty,
 \end{equation}
 where $\omega_f(t):=\sup_{x,y\in\Omega,\,|x-y|\leq t} |f(x)-f(y)|$.
 %modulus of continuity $f$ satisfying the following integrability condition and there exists;
\end{itemize}

In some instances in place of (H3) we will require the stronger condition
\begin{itemize}
 \item[(H4)] $f>c_0$ $\Ln$ a.e. $\Omega$, for some $c_0>0$, and
 $f$ is double-Dini continuous, that is   
  \begin{equation}\label{H4}
    \int_0^1 \frac{\omega_f(r)}{r}\,|\log r|^a\, dr < \infty,
  \end{equation}
for some $a\geq 1$. 
%  and exists $c_0>0$ such that $f\geq c_0$;
\end{itemize}
Note that for the zero obstacle problem, i.e.~$\psi=0$, assumptions (H3) and (H4) involve only the lower 
order term $h$ in the integrand and not the matrix field $\mathbb{A}$. 

Given the assumptions introduced above we provide a full free boundary stratification result. 
\begin{theorem}\label{t:linear}
Assume (H1)-(H4) to hold, and let $w$ be the (unique) minimizer of $\mathcal{E}$ in \eqref{e:enrg}
on $\mathcal{K}_{\psi,g}$.

Then, $w$ is $W^{2,p}_{loc}\cap C^{1,{1-\sfrac{n}{p}}}_{loc}(\Omega)$, and the free boundary 
decomposes as $\partial \{w = \psi\} \cap \Omega = \Reg(w) \cup \Sing(w)$, where 
$\Reg(w)$ and $\Sing(w)$ are called its regular and singular part, respectively. 
Moreover, $\Reg(w) \cap \Sing(w) = \emptyset$ and
\begin{itemize}
\item[(i)] if $a>2$ in (H4), then $\Reg(w)$ is relatively open in $\partial \{w = \psi\}$
and, for every point $x_0 \in \Reg(w)$, there exist $r=r(x_0)>0$ such that $\partial \{w = \psi\}\cap B_r(x)$ 
is a $C^1$ $(n-1)$-dimensional manifold with normal vector absolutely continuous. 
 
 In particular if $f$ is H\"{o}lder continuous there exists $r>0$ such that $\partial \{w = \psi\}\cap B_r(x)$ 
 is a $C^{1,\beta}$ $(n-1)$-dimensional manifold for some exponent $\beta\in (0,1)$.

\item[(ii)] if $a\geq1$ in (H4), then $\Sing(w) = \cup_{k=0}^{n-1} S_k$, with $S_k$ contained in the union of
at most countably many submanifolds of dimension $k$ and class $C^1$.
\end{itemize}
\end{theorem}
Theorem~\ref{t:linear} has been proved by Caffarelli for suitably regular matrix fields,
and it is the resume of his long term program on the subject (cf. for instance \cite{Caf77, Caf80, Caf98-Fermi, Caf98} 
and the books \cite{KS, CS, PSU} for more details and references also on related problems). 
Let us also remark that very recently the fine structure of the set of singular points for the Dirichlet energy has been unveiled in the 
papers by Colombo, Spolaor and Velichkov \cite{CoSpVe17} and 
Figalli and Serra \cite{FiSe17} by means of a logarithmic 
epiperimetric inequality and new monotonicity formulas, respectively.  

In the last years Theorem~\ref{t:linear} has been extended to 
the case in which $\mathbb{A}$ either is
Lipschitz continuous in \cite{FoGeSp15} or belongs to a fractional Sobolev space $W^{1+s,p}$  in \cite{Ger17}, with $s$, $p$ and 
$n$ suitably related, and also in some nonlinear cases \cite{FoGerSp17}. The last papers follow the variational approach 
to free boundary analysis developed 
remarkably by Weiss \cite{Weiss} and by Monneau \cite{Monneau03}.
The extensions of Weiss' and Monneau's quasi-monotonicity formulas obtained in the papers \cite{FoGeSp15,Ger17} 
hinge upon a generalization of the Rellich and Ne\v{c}as' inequality due to Payne and Weinberger (cf. \cite{Kukavica}). 
On a technical side they involve the derivation of the matrix field $\mathbb{A}$. 
The main difference contained in the present note with respect to the papers \cite{FoGeSp15,Ger17}  
concerns the monotone quantity itself. Indeed, rather than considering the natural quadratic energy associated to 
the obstacle problem under study, we establish quasi-monotonicity for a related constant coefficient quadratic form. 
The latter result is obtained thanks to a freezing argument inspired by some computations of Monneau 
(cf. \cite[Section~6]{Monneau03}) in combination with the well-known quadratic lower bound on the growth of solutions 
from free boundary points (see Section~\ref{s:q-mon formula} for more details).
Such an insight, though elementary, has been overlooked in the literature and enables us to obtain
Weiss' and Monneau's quasi-monotonicity formulas under the milder assumptions (H1) and (H3) (the latter
having no role if $\psi=0$), since the matrix field $\mathbb{A}$ is not differentiated along the derivation process 
of the quasi-monotonicity formulas.

To conclude this introduction we briefly resume the structure of the paper: standard preliminaries for the classical obstacle problem 
are collected in Section~\ref{s:prel}. 
The mentioned generalizations of Weiss' and Monneau's quasi-monotonicity formulas are dealt with in Section~\ref{s:q-mon formula}, 
finally Section~\ref{s:applications} contains the applications to the free boundary stratification for quadratic problems.

%%%%%%%%%%%%%%%%%%%%%%%%%%%%%%%%%%%
%%%%%%%%%%%%%%%%%%%%%%%%%%%%%%%%%%%
%%%%%%%%%%%%%%%%%%%%%%%%%%%%%%%%%%%
%%%%%%%%%%%%%%%%%%%%%%%%%%%%%%%%%%%
%%%%%%%%%%%%%%%%%%%%%%%%%%%%%%%%%%%
%%%%%%%%%%%%%%%%%%%%%%%%%%%%%%%%%%%
%%%%%%%%%%%%%%%%%%%%%%%%%%%%%%%%%%%
%
%	SECTION 1
%
%%%%%%%%%%%%%%%%%%%%%%%%%%%%%%%%%%%
%%%%%%%%%%%%%%%%%%%%%%%%%%%%%%%%%%%
%%%%%%%%%%%%%%%%%%%%%%%%%%%%%%%%%%%
%%%%%%%%%%%%%%%%%%%%%%%%%%%%%%%%%%%
%%%%%%%%%%%%%%%%%%%%%%%%%%%%%%%%%%%
%%%%%%%%%%%%%%%%%%%%%%%%%%%%%%%%%%%
%%%%%%%%%%%%%%%%%%%%%%%%%%%%%%%%%%%
\section{Preliminaries}\label{s:prel}

% 
% \begin{proposition}
% There exists a unique solution for the minimum problem 
% \begin{equation}\label{p:problema ad ostacolo}
% \inf_K \mathscr{E}[\cdot],
% \end{equation}
% where 
% \begin{equation}
%  K:=\{v\in W^{1,2}(\Omega)\, |\,  v\geq 0\, \Ln\textit{-a.e. on}\,\, \Omega,\, \mathrm{Tr}(v)=g\,\textit{on}\,\, \partial\Omega \},
% \end{equation}
% $g\in H^\frac{1}{2}(\partial\Omega)$ being a nonnegative function.
% \end{proposition}
Throughout the section we use the notation introduced in Section~\ref{s:intro} and adopt Einstein' summation convention.

We first reduce ourselves to the zero obstacle problem. 
Let $w$ be the unique minimizer of $\mathcal{E}$ over $\mathcal{K}_{\psi,g}$, 
and define $u:=w-\psi$. Then, $u$ is the unique minimizer of 
 \begin{equation}\label{e:enrg2}
 \mathscr{E}(v):=\int_\Omega \big(\langle \mathbb{A}(x)\nabla v(x),\nabla v(x)\rangle +2f(x)v(x)\big)\,dx,
\end{equation}
over
\[
\mathbb{K}_{\psi,g} := \big\{v\in W^{1,2}(\Omega):\,v\geq 0\;\;\Ln\,\text{ a.e. on } \Omega,\,
\textup{Tr}(v)=g-\psi \,\text{ on } \partial\Omega\big\}\,,
\]
where $f=h-\diw(\mathbb{A}\nabla\psi)$. Clearly, $\partial\{w=\psi\}\cap\Omega=\partial\{u=0\}\cap\Omega$, therefore we shall establish 
all the results in Theorem~\ref{t:linear} for $u$ (notice that assumptions (H3) and (H4) are formulated exactly in terms of $f$).

Note that $u$ satisfies a PDE both in the distributional sense and a.e. on $\Omega$,
elliptic regularity then applies to $u$ itself to establish its smoothness.
The next result had been established by Ural'tseva in \cite{Ural87}
with a different proof.
\begin{proposition}\label{p:PDE u}
Let $u$ be the minimum of $\mathscr{E}$ on $\mathbb{K}_{\psi,g}$. Then 
 \begin{equation}\label{PDE_u}
  \mathrm{div}(\mathbb{A}\nabla u)=f\chi_{\{u>0\}} %\qquad \textit{$\Ln$-a.e. on $\Omega$ and in $\mathcal{D}'(\Omega)$}. 
 \end{equation}
$\Ln$-a.e. on $\Omega$ and in $\mathcal{D}'(\Omega)$.
 Moreover, $u\in W^{2,p}_{loc}\cap C^{1,1-\frac{n}{p}}_{loc}(\Omega)$.
\end{proposition}
\begin{proof}
For the validity of \eqref{PDE_u} we refer to 
\cite[Proposition 3.2]{FoGerSp17} where the result is proven in the broader
context of variational inequalities (see also  \cite[Proposition 2.2]{FoGeSp15}).

From this, by taking into account that 
$\mathbb{A}\in C^{0,1-\sfrac{n}{p}}_{loc}(\Omega,\R^{n\times n})$
in view of Morrey embedding theorem, Schauder estimates then yields 
$u\in C_{loc}^{1,1-\sfrac{n}{p}}(\Omega)$ (cf. \cite[Theorem 3.13]{HL}). 

Next consider the equation
 \begin{align}\label{PDE u Miranda}
  %\sum_{ij}
  a_{ij}\, \frac{\partial^2v}{\partial x_i\partial x_j} = f\chi_{\{u>0\}}-%\sum_{j}
  \mathrm{div}\mathbb{A}^j\frac{\partial u}{\partial x_j}=:\varphi,
 \end{align}
% where the symbol $v_{ij}$ is the partial second derivative in the variable $x_i$ and $x_j$ and 
where $\mathbb{A}^j$ denotes the $j$-column of $\mathbb{A}$. 
 Since $\nabla u\in L^\infty_{loc}(\Omega,\R^n)$ and $\diw \mathbb{A}^j\in L^{p}(\Omega)$ for all $j\in\{1,\dots,n\}$,
 then $\varphi\in L^{p}_{loc}(\Omega)$. 
 \cite[Corollary 9.18]{GT} implies the uniqueness of a solution 
 $v\in W_{loc}^{2,p}(\Omega)$ to \eqref{PDE u Miranda}. 
 By taking into account the identity 
 $\mathrm{Tr}(\mathbb{A}\nabla^2 v)=\mathrm{div}(\mathbb{A}\nabla v)-%\sum_{j}
 \mathrm{div}\mathbb{A}^j\frac{\partial v}{\partial x_j}$, \eqref{PDE u Miranda} rewrites as 
\begin{equation}
 \mathrm{div}(\mathbb{A}\nabla v)-%\sum_{j}
 \mathrm{div}\mathbb{A}^j\frac{\partial v}{\partial x_j} =\varphi,
\end{equation}
we have that $u$ and $v$ are two solutions. Then by \cite[Theorem 1.I]{Mir63} we obtain $u=v$.
\end{proof}

%\subsection{Preliminary free boundary analysis}

We recall next the standard notations for the coincidence set and 
for the free boundary 
 \begin{equation}
  \Lambda_u=\{x\in\Omega:\,u(x)=0\}\,,
%  \qquad\qquad \Omega\setminus\Lambda_u=\{u>0\}\cap \Omega 
  \qquad\Gamma_u =\partial\Lambda_u\cap \Omega.
 \end{equation}
For any point $x_0\in \Gamma_u$, we introduce the family of rescaled functions
\begin{equation}\label{u_x_0 r}
u_{x_0,r}(x):=\frac{u(x_0+rx)}{r^2}
\end{equation}
for $x\in\frac 1r(\Omega-\{x_0\})$. The existence of $C^{1,\alpha}$-limits as 
$r\downarrow 0$ of the latter family is standard by noting that the rescaled functions 
satisfy an appropriate PDE and then uniform $W^{2,p}$ estimates. 

\begin{proposition}[{\cite[Proposition 4.1]{Ger17}}]\label{prop u_r limitata W2p}
 Let $u$ be the unique minimizer of $\mathscr{E}$ over $\mathbb{K}_{\psi,g}$, and $K\subset\Omega$ a compact set.
Then for every $x_0\in K\cap\Gamma_u$, for every $R>0$
 there exists a constant $C=C(n,p,\Lambda,R,K,\|f\|_{L^\infty},
 \|\mathbb{A}\|_{W^{1,p}})>0$ such that, for every $r\in (0,\frac{1}{4R}\mathrm{dist}(K, \partial\Omega))$
 \begin{equation}\label{u_r limitata W2p}
  \|u_{x_0,r}\|_{W^{2,p}(B_R)}\leq C.
 \end{equation}
In particular, $(u_{x_0,r})_r$ is equibounded in $C^{1,\gamma}_{\mathrm{loc}}$ for $\gamma\in(0, 1-\sfrac{n}{p}]$.
\end{proposition}
The functions arising in this limit process are called blow-up limits.
\begin{corollary}[Existence of blow-ups]
Let $u$ be the unique minimizer of $\mathscr{E}$ over $\mathbb{K}_{\psi,g}$, and let $x_0\in \Gamma_u$. 
Then, for every sequence $r_k\downarrow 0$ there exists a subsequence $(r_{k_j})_j\subset (r_k)_k$ such 
that the rescaled functions $(u_{x_0,r_{k_j}})_j$ converge in $C^{1,\gamma}_{\mathrm{loc}}$, $\gamma\in(0,1-\sfrac np)$.
 %We define these limits as \emph{blow-ups}.
\end{corollary}

Elementary growth conditions of the solution from free boundary points are easily deduced from Proposition~\ref{prop u_r limitata W2p} 
and the condition $p>n$. In turn, such properties will be crucial in the derivation of the quasi-monotonicity formulas. 
\begin{proposition}\label{p:esitmate D2u}
  Let $u$ be the unique minimizer of $\mathscr{E}$ over $\mathbb{K}_{\psi,g}$. % and let $x_0\in \Gamma_u$. 
 Then for all compact sets $K\subset \Omega$ there exists a constant $C=C(n,p,\Lambda,K,\|f\|_{L^\infty},
 \|\mathbb{A}\|_{W^{1,p}})>0$ %(independent of $r$) 
 such that for all points $x_0\in \Gamma_u\cap K$, and for all $r\in \big(0, \frac{1}{2}\mathrm{dist}(K, \partial\Omega)\big)$ 
 it holds
\begin{equation}\label{e:u e grad u limitate}
 \|u\|_{L^\infty(B_r(x_0))}\leq C\, r^2\,, 
 \qquad \|\nabla u\|_{L^\infty(B_r(x_0),\R^n)}\leq C\, r.   
\end{equation}
 and
\begin{equation}\label{e:estimate D2u}
  \|\nabla^2 u\|_{L^p(B_r(x_0),\R^{n\times n})}
  \leq C\, r^{\sfrac{n}{p}}.
\end{equation}
\end{proposition}
%\begin{proof}
%Since for $x_0\in \Gamma_u$ we have $u(x_0)=0$ and 
%$\nabla u(x_0)=\underline{0}$, a scaling argument and 
%Proposition~\ref{prop u_r limitata W2p} give 
%\eqref{e:u e grad u limitate}.
%Since the constant in \eqref{e:u e grad u limitate} only depends on the constant appearing in \eqref{u_r limitata W2p}, it is 
%uniformly bounded for $x_0\in \Gamma_u\cap K$, $K\subset\Omega$ compact.
%
%%In turn the estimates in \eqref{e:u e grad u limitate} yield an analogous estimate of the second derivative of $u$ near free boundary points. Indeed, 
%Finally, thanks to Proposition~\ref{prop u_r limitata W2p} and by rescaling we have
%\begin{equation}
% \begin{split}
%  \int_{B_r(x_0)}|\nabla^2 u(x)|^p\,dx \stackrel{x=x_0 +ry}{=} r^n \int_{B_1}|\nabla^2 u(x_0+ry)|^p\,dy =r^n \int_{B_1}|\nabla^2 u_r(y)|^p\,dy
%  \leq C\,r^n.
% \end{split}
%\end{equation}
%By passing to the $L^p$-norm we conclude the proof.
%\end{proof}

Finally, we recall the fundamental quadratic detachment property from free boundary points 
that entails non triviality of blow up limits. 
It has been established by Blank and Hao  in \cite[Theorem~3.9]{BlankHao15} under the sole $VMO$ regularity of $\mathbb{A}$,
an assumption weaker than (H1).
\begin{lemma}\label{l:quadratic}
There exists a constant $\vartheta=\vartheta(n,\Lambda,c_0,\|f\|_{L^\infty})>0$ such that, for
every $x_0 \in \Gamma_u$ and $r\in(0,\frac12\textup{dist}(x_0,\partial\Omega))$, it holds
\begin{equation*}%\label{e:quadratic}
\sup_{x \in \partial B_r(x_0)} u(x) \geq \vartheta\,r^2.
\end{equation*}
\end{lemma}

\section{Quasi-monotonicity formulas}\label{s:q-mon formula}

In this section we establish Weiss' and Monneau's type quasi-monotonicity formulas for the quadratic problem. 
As pointed out in Section~\ref{s:intro} the main difference with the existing literature concerns the monotone 
quantity itself. Indeed, rather than considering the natural quadratic energy $\mathscr{E}$ associated 
to the obstacle problem under study, we may consider the classical Dirichlet energy thanks to a normalization. 
In doing this we have been inspired by Monneau \cite[Section~6]{Monneau03}. The advantage 
of this formulation is that the matrix field $\mathbb{A}$ is not
differentiated in deriving the quasi-monotonicity formulas contrary to \cite{FoGerSp17, Ger17}. 
Our additional insight is elementary but crucial: we further exploit the 
quadratic growth of solutions from free boundary points in Proposition~\ref{p:esitmate D2u} to establish quasi-monotonicity.
%This simple observation seems to have been overlooked in the literature.

Let $x_0 \in \Gamma_u$ be any point of the free boundary, then the affine change of variables
\[
x\mapsto x_0+f^{-\sfrac12}(x_0)\mathbb{A}^{\sfrac12}(x_0)x =: x_0 + \LL(x_0)\, x
\]
leads to
\begin{equation}\label{e:cambio di coordinate3}
\mathscr{E}(u)=f^{1-\frac{n}{2}}(x_0)\det(\mathbb{A}^{\sfrac12}(x_0))\,\mathscr{E}_{\LL(x_0)}(u_{\LL(x_0)}),
\end{equation}
where $\Omega_{\LL(x_0)}:=\LL^{-1}(x_0)\,(\Omega-x_0)$, and we have set
\begin{equation}\label{e:enrgA}
\mathscr{E}_{\LL(x_0)}(v):=\int_{\Omega_{\LL(x_0)}}\left(\langle \mathbb{C}_{x_0}\nabla v,\nabla v\rangle 
+ 2\frac{f_{\LL(x_0)}}{f(x_0)}\,v\right)dx,
\end{equation}
with
\begin{gather}
u_{{\LL(x_0)}}(x)  :=u\big(x_0+\LL(x_0)x\big), 
\label{e:cambio di coordinate1}\\
f_{{\LL(x_0)}}(x)  :=f\big(x_0+\LL(x_0)x\big), \notag
%\label{e:cambio di coordinate2}
\\
{\mathbb C}_{x_0}(x)  :=
\mathbb{A}^{-\sfrac12}(x_0)\mathbb{A}(x_0+\LL(x_0)x)\mathbb{A}^{-\sfrac12}(x_0).\notag
\end{gather}
Note that $f_{\LL(x_0)}(\underline{0})=f(x_0)$ and ${\mathbb C}_{x_0}(\underline{0})=\mathrm{Id}$.
Moreover, the free boundary is transformed under this map into
\[
\Gamma_{u_{\LL(x_0)}}=\LL^{-1}(x_0)(\Gamma_u-x_0),
\] 
and the energy $\mathscr{E}$ in \eqref{e:enrg} is minimized by $u$ if and only if 
$\mathscr{E}_{\LL(x_0)}$ in \eqref{e:enrgA} is minimized by the function $u_{\LL(x_0)}$
in \eqref{e:cambio di coordinate1}.

In addition, rewriting the Euler-Lagrange equation for $u_{\LL(x_0)}$ in non-divergence 
form we get on $\Omega_{\LL(x_0)}$: 
\begin{equation*}%\label{e:nondivergence}
 c_{ij}(x)\,\frac{\partial^2u_{\LL(x_0)}}{\partial x_i\partial x_j} 
 + \mathrm{div}\mathbb{A}^i(x)\, \frac{\partial u_{\LL(x_0)}}{\partial x_i} %+ c(x)\, u 
 = \frac{f_{\LL(x_0)}(x)}{f(x_0)}\chi_{\{u_{\LL(x_0)}>0\}}\,.
\end{equation*}
(using again Einstein's convention) with ${\mathbb C}_{x_0}=(c_{ij})_{i,j=1,\ldots,n}$.
Moreover, we may further rewrite the latter equation on $\Omega_{\LL(x_0)}$ as
\begin{equation}\label{e:PDE u f1}
\begin{split}
 \Delta  u_{\LL(x_0)} &= 1
 + \Big( \frac{f_{\LL(x_0)}(x)}{f(x_0)}\chi_{\{u_{\LL(x_0)}>0\}}-1 
 -\big(c_{ij}(x)-\delta_{ij}\big)\frac{\partial^2u_{\LL(x_0)}}{\partial x_i\partial x_j} 
 -\diw\mathbb{C}^i_{x_0}(x)\, \frac{\partial u_{\LL(x_0)}}{\partial x_i}\Big)\\
 &=: 1+f_{x_0}(x)\,.
\end{split}
\end{equation}
We are now ready to establish Weiss' and Monneau's quasi-monotonicity formulas for $u$ by using equality
\eqref{e:PDE u f1} and Proposition \ref{p:esitmate D2u}.

\subsection{Weiss' quasi-monotonicity formula}

In this section we consider the Weiss' energy 
 \begin{equation}\label{e:Weiss energy}
  \Phi_{u}(x_0,r):=\frac{1}{r^{n+2}}\int_{B_r} \big(|\nabla u_{\LL(x_0)}|^2 + 2\,u_{\LL(x_0)}\big)\,dx 
  - \frac{2}{r^{n+3}}\int_{\partial B_r} u_{\LL(x_0)}^2\, d\mathcal{H}^{n-1}\,,
 \end{equation}
 $x_0\in\Gamma_u$, and prove its quasi-monotonicity. %(recall that we are working under assumption (H5)).
\begin{theorem}[Weiss' quasi-monotonicity formula]\label{t:Wp>n} 
Under assumptions (H1)-(H3), if $K\subset\Omega$ is a compact set, there is a constant $C=C(n,p,\Lambda,c_0,K,\|f\|_{L^\infty},\|\mathbb{A}\|_{W^{1,p}})>0$ 
such that for all $x_0\in K\cap \Gamma_u$
\begin{equation}\label{e:Weiss}
\frac{d}{dr}\Big(\Phi_{u}(x_0,r) + C\int_0^r\frac{\omega(t)}{t}\,dt\Big)\geq 
\frac{2}{r^{n+4}}\int_{\partial B_r} (\langle\nabla u_{\LL(x_0)}, x\rangle-2u_{\LL(x_0)})^2 d\mathcal{H}^{n-1},
\end{equation}
for $\mathcal{L}^1$ a.e. $r\in(0,\frac12\mathrm{dist}(K,\partial\Omega))$,
where $\omega(r):=\omega_f(r)+r^{1-\frac{n}{p}}$. 
%is Dini-continuous.

In particular, $\Phi_{u}(x_0,\cdot)$ has finite right limit $\Phi_{u}(x_0,0^+)$ in zero, and 
for all $r\in(0,\frac12\mathrm{dist}(K,\partial\Omega))$,
\begin{equation}\label{e:Phi(r)-Phi(0)}
 \Phi_{u}(x_0,r)-\Phi_{u}(x_0,0^+)\geq -C\int_0^r\frac{\omega(t)}{t}\,dt.
\end{equation} 
\end{theorem}
\begin{proof}[Proof of Theorem~\ref{t:Wp>n}]
We analyse separately the volume and the boundary terms
appearing in the definition of the Weiss energy in 
\eqref{e:Weiss energy}. For the sake of notational simplicity 
we write $u_{x_0}$ in place of $u_{\LL(x_0)}$.
In what follows $C=C(n,p,\Lambda,c_0,K,\|f\|_{L^\infty},
 \|\mathbb{A}\|_{W^{1,p}})>0$ denotes a constant that 
 may vary from line to line.

We start off with the bulk term. The Coarea Formula implies for $\mathcal{L}^1$-a.e. 
$r\in(0,\mathrm{dist}(K,\partial\Omega))$  
\begin{align}\label{e:Phi' add1}
\frac{d}{dr}&\Big(\frac{1}{r^{n+2}}\int_{B_r} \left(|\nabla u_{x_0}|^2 + 2\,u_{x_0}\right)\,dx\Big)=\notag\\ 
&-\frac{n+2}{r^{n+3}}\int_{B_r}\left(|\nabla u_{x_0}|^2 + 2\,u_{x_0}\right)\,dx + \frac{1}{r^{n+2}}\int_{\partial B_r}\left(|\nabla u_{x_0}|^2 + 2\,u_{x_0}\right)\,dx.
\end{align}
We use the Divergence Theorem together with the following
identities
\[
|\nabla u_{x_0}|^2=
\frac{1}{2}\diw(\nabla(u^2_{x_0})) - u_{x_0}\,\Delta u_{x_0}\,,
\]
\[
  \diw\left(|\nabla u_{x_0}|^2\,\frac{x}{r}\right)
  =\frac{n-2}{r}|\nabla u_{x_0}|^2 
  - 2\Delta u_{x_0} \langle\nabla u_{x_0},\frac{x}{r}\rangle 
  + 2\,\diw\left(\langle\nabla u_{x_0},\frac{x}{r}\rangle
  \nabla u_{x_0}  \right),
 \]
\[
  \diw\left(u_{x_0}\,\frac{x}{r}\right)= u_{x_0}\,\frac{n}{r} + \langle\nabla u_{x_0},\frac{x}{r}\rangle,
\]
to deal with the first, third and fourth addend in 
\eqref{e:Phi' add1}, respectively.
Hence, we can rewrite the right hand side of equality \eqref{e:Phi' add1} as follows
\begin{align}\label{e:Phi' add11}
\frac{d}{dr}&\Big(\frac{1}{r^{n+2}}\int_{B_r} \left(|\nabla u_{x_0}|^2 + 2\,u_{x_0}\right)\,dx\Big)
= \frac{2}{r^{n+2}}\int_{B_r}(\Delta u_{x_0}-1)\left(2\,\frac{u_{x_0}}{r}-\langle\nabla u_{x_0},\frac{x}{r}\rangle\right)\,dx \notag\\
&+\frac{2}{r^{n+2}}\int_{\partial B_r} \langle\nabla u_{x_0},\frac{x}{r}\rangle^2 d\mathcal{H}^{n-1} 
-\frac{4}{r^{n+2}}\int_{\partial B_r} \frac{u_{x_0}}{r}\,\langle\nabla u_{x_0},\frac{x}{r}\rangle d\mathcal{H}^{n-1}\,.
\end{align}

We consider next the boundary term in the expression of $\Phi_u$. By scaling and a direct calculation we get
\begin{align}\label{e:Phi' add2}
 \frac{d}{dr}\Big(\frac{2}{r^{n+3}}\int_{\partial B_r} u^2_{x_0}\,d\mathcal{H}^{n-1}\Big)
 &\stackrel{x=ry}{=}2\int_{\partial B_1} \frac{d}{dr} \left(\frac{u_{x_0}(ry)}{r^2}\right)^2\, d\mathcal{H}^{n-1}\notag\\
 &=4\int_{\partial B_1}\frac{u_{x_0}(ry)}{r^4}\left(\langle\nabla u_{x_0}(ry),y\rangle -2\frac{u_{x_0}(ry)}{r}\right)\, d\mathcal{H}^{n-1}\notag\\
 &\stackrel{x=ry}{=} \frac{4}{r^{n+2}} \int_{\partial B_r} \frac{u_{x_0}}{r}\,\langle\nabla u_{x_0},\frac{x}{r}\rangle\, d\mathcal{H}^{n-1}
 - \frac{8}{r^{n+2}} \int_{\partial B_r}\frac{u^2_{x_0}}{r^2}\, d\mathcal{H}^{n-1}\,.
 \end{align}
Then by combining together the equations \eqref{e:Phi' add11} and \eqref{e:Phi' add2} and recalling equation \eqref{e:PDE u f1}
we obtain
\begin{equation*}%\label{e:stima finale}
\begin{split}
 \Phi'_u(x_0,r) &= \frac{2}{r^{n+2}}\int_{B_r}f_{x_0}\,\left(2\,\frac{u_{x_0}}{r}-\langle\nabla u_{x_0},\frac{x}{r}\rangle\right)\,dx 
+\frac{2}{r^{n+2}}\int_{\partial B_r} \left(\langle\nabla u_{x_0},\frac{x}{r}\rangle - 2\,\frac{u_{x_0}}{r}\right)^2 d\mathcal{H}^{n-1}\\
&= \frac{2}{r^{n+2}}\int_{B_r\setminus \Lambda_{u_{x_0}}}f_{x_0}\,\left(2\,\frac{u_{x_0}}{r}-\langle\nabla u_{x_0},\frac{x}{r}\rangle\right)\,dx 
+\frac{2}{r^{n+2}}\int_{\partial B_r} \left(\langle\nabla u_{x_0},\frac{x}{r}\rangle - 2\,\frac{u_{x_0}}{r}\right)^2 d\mathcal{H}^{n-1},
\end{split}
\end{equation*}
 where in the last equality we used the unilateral obstacle condition to deduce that $\Lambda_{u_{x_0}}\subseteq\{\nabla u_{x_0} =\underline{0}\}$.
 Therefore, by the growth of $u$ and $\nabla u$ 
 from $x_0$ in \eqref{e:u e grad u limitate} we obtain
\begin{equation}\label{e:phiprimo}
 \Phi'_u(x_0,r) \geq -\frac{C}{r^{n+1}}\int_{B_r\setminus \Lambda_{u_{x_0}}}|f_{x_0}|\,dx + \frac{2}{r^{n+2}}\int_{\partial B_r} \left(\langle\nabla u_{x_0},\frac{x}{r}\rangle 
 - 2\,\frac{u_{x_0}}{r}\right)^2 d\mathcal{H}^{n-1}\,.
\end{equation}
Next note that by (H1), (H3), and by the very definition of $f_{x_0}$ in \eqref{e:PDE u f1} it follows that
\begin{equation}\label{e:uno}
\frac{1}{r^{n+1}}
\int_{B_r\setminus \Lambda_{u_{x_0}}}|f_{x_0}|\,dx 
\leq\frac{\omega_f(r)}{c_0\,r}+
 \frac C{r^{n(1+\frac 1p)}}\int_{B_r}|\nabla^2 u_{x_0}|\,dx + 
 \frac C{r^n}\int_{B_r}|\diw\mathbb{C}_{x_0}|\,dx\,.
\end{equation}
%for some $C=C(n,p,K,\|\mathbb{A}\|_{C^{0,1-\sfrac np}})>0$.
 By \eqref{e:estimate D2u} we estimate the second addend on the right hand side of the last inequality as follows 
\begin{equation}\label{e:due}
 \frac1{r^{n(1+\frac1p)}}
 \int_{B_r}|\nabla^2 u_{x_0}|\,dx 
 \leq  \frac C{r^{n(1+\frac1p)}}
 \|\nabla^2 u_{x_0}\|_{L^p(B_r,\R^{n\times n})}
 (\omega_n r^n)^{1-\frac{1}{p}}\leq C\, r^{-\frac{n}{p}}\,, 
\end{equation}
%for some constant $C=C(n,p,K)>0$, and 
by H\"older inequality we get for the third addend
\begin{equation}\label{e:tre}
 \frac{1}{r^n} \int_{B_r}|\diw\mathbb{C}_{x_0}|\,dx
 \leq \frac{1}{r^n}\|\diw\mathbb{C}_{x_0}\|_{L^p(B_r,\R^n)}\,(\omega_n r^n)^{1-\frac{1}{p}}
 \leq C\, r^{-\frac{n}{p}}\,.
\end{equation}
%for some constant  $C=C(n,p,\Lambda,K,\|\diw\mathbb{A}\|_{L^p})>0$.
Therefore, we conclude from \eqref{e:phiprimo}-\eqref{e:tre} 
\begin{equation*}
\begin{split}
 \Phi'_u(x_0,r) &\geq 
 %-c \left(\frac{\omega_f(r)}{r}+\frac{\omega_{\mathbb{A}}(r)}{r} + r^{-\frac{n}{p}}+ r\sup_{B_r}c\right)+ \frac{2}{r^{n+4}}\int_{\partial B_r} \left(\langle\nabla u,\frac{x}{r}\rangle - 2\,\frac{u}{r}\right)^2 d\mathcal{H}^{n-1}\\
 -C\,\frac{\omega(r)}{r} + \frac{2}{r^{n+2}}\int_{\partial B_r} \left(\langle\nabla u_{x_0},\frac{x}{r}\rangle - 2\,\frac{u_{x_0}}{r}\right)^2 d\mathcal{H}^{n-1}\,,
\end{split}
\end{equation*}
where $\omega(r):=\omega_f(r)+ r^{1-\frac{n}{p}}$. 
%and $C=C(n,p,c_0,\Lambda,K,\|\mathbb{A}\|_{W^{1,p}})>0$.
\end{proof}
\begin{remark}
Recalling that $f$ %and $\mathbb{A}$ are 
is Dini-continuous by (H3), the modulus of continuity 
$\omega$ provided by Theorem~\ref{t:Wp>n} is in 
turn Dini-continuous.
\end{remark}

\begin{remark}
More generally, the argument in Theorem~\ref{t:Wp>n} works for solutions to second order elliptic PDEs 
in nondivergence form of the type
\[%\begin{equation}
 a_{ij}(x)\,u_{ij} + b_i(x)\, u_i + c(x)\, u 
 = f(x)\chi_{\{u>0\}}\,,
\]%\end{equation}and rewrite its as follows:
% provided we assume $\mathbb{A}$ to be Dini-continuous with $\diw\mathbb{A}\in L^p(\Omega,\R^{n\times n})$,
% instead of (H1).
% Indeed, \cite[Theorem~3.9]{BlankHao15} ensures the two estimates in formula \eqref{e:u e grad u limitate}, instead for what 
% \eqref{e:estimate D2u} is concerned
the only difference with the statement of Theorem~\ref{t:Wp>n} being that in this framework
$\omega(r):=\omega_f(r)%+\omega_{\mathbb{A}}(r)
+ r^{1-\frac{n}{p}}+ r^2\sup_{B_r}c$.
\end{remark}

\subsection{Monneau's quasi-monotonicity formula}
Let $v$ be a positive $2$-homogeneus polynomial solution of 
\begin{equation}\label{d:v}
 \Delta v = 1\quad \textrm{on $\R^n$}.
\end{equation}
Then by $2$-homogeneity, elementary calculations lead to 
 \begin{equation}\label{e:Phiv}
\Phi_v(\underline{0},r)=\Phi_v(\underline{0},1)= \int_{B_1} v\, dy,
 \end{equation}
for all $r>0$.
We prove next a quasi-monotonicity formula for solutions of the obstacle problem in case $x_0\in\Gamma_u$ 
is a singular point of the free boundary, namely it is such that 
\begin{equation}\label{e:en BU sing Monneau}
 \Phi_u(x_0,0^+)=\Phi_v(\underline{0},1)
\end{equation}
for some $v$ $2$-homogeneous solution of \eqref{d:v}.
\begin{theorem}[Monneau's quasi-monotonicity formula]\label{t:Mp>n}
Under hypotheses (H1), (H2), (H4) with $a=1$, if $K\subset\Omega$ is a compact set and 
\eqref{e:Phiv} holds for $x_0\in K\cap \Gamma_u$ then there exists a constant $C=C(n,p,\Lambda,c_0,K,\|f\|_{L^\infty},\|\mathbb{A}\|_{W^{1,p}})>0$ such that the function 
 \begin{equation}\label{e:Monneau}
  \big(0,\frac12 \mathrm{dist}(K,\partial\Omega)\big)\ni r\longmapsto 
  \frac{1}{r^{n+3}}\int_{\partial B_r} ( u_{\LL(x_0)}-v)^2\,dx + %\frac{C}{r}\int_0^r\frac{\omega(t)}{t}\,dt
C\int_0^r\frac{dt}{t}\int_0^t\frac{\omega(s)}{s}\,ds
 \end{equation}
is nondecreasing, where $v$ is any $2$-homogeneus polynomial solution of \eqref{d:v}, and $\omega$ is the modulus of
continuity provided by Theorem~\ref{t:Wp>n}. 
%In particular,
% \begin{equation}
%  \frac{d}{dr}\Big(\frac{1}{r^{n+3}}\int_{\partial B_r} (u-v)^2 + C\int_0^r\frac{dt}{t}\int_0^t\frac{\omega(s)}{s}\,ds\Big)\geq 0.
% \end{equation}
\end{theorem}
% \begin{remark}
%  In particular, if $K\subset\Omega$ is compact, formula \eqref{e:Monneau} 
%  holds for every $x_0\in K$ and for all  $r\in(0,\frac12\mathrm{dist}(K,\partial\Omega))$ with a 
%  uniform constant $C$.
% \end{remark}
\begin{proof}[Proof of Theorem~\ref{t:Mp>n}]
As in the proof of Theorem~\ref{t:Wp>n} for the sake of notational simplicity we write $u_{x_0}$ rather than $u_{\LL(x_0)}$. 
 
 Set $w:=u_{x_0}-v$, then arguing as in \eqref{e:Phi' add2} and by applying the Divergence Theorem we get
\begin{equation}\label{e:monneau1}
\begin{split}
 \frac{d}{dr}&\left(\frac{1}{r^{n+3}}\int_{\partial B_r} 
 w^2\, d\mathcal{H}^{n-1}\right)
 = \frac{2}{r^{n+3}}\int_{\partial B_r}w\left(\langle\nabla w,\frac{x}{r}\rangle -2\frac{w}{r}\right)\, d\mathcal{H}^{n-1}\\
  &= \frac{2}{r^{n+3}}\int_{B_r} \diw(w\nabla w)\,dx - \frac{4}{r^{n+4}}\int_{\partial B_r} w^2\, d\mathcal{H}^{n-1}\\
  &= \frac{2}{r^{n+3}}\int_{B_r} w\Delta w\, dx + \frac{2}{r^{n+3}}\int_{B_r}|\nabla w|^2\,dx - \frac{4}{r^{n+4}}\int_{\partial B_r} w^2\, d\mathcal{H}^{n-1}\\
\end{split}
\end{equation}
For what concerns the first term on the right hand side of \eqref{e:monneau1} recall that $u\in W^{2,p}_{loc}(\Omega)$, 
thus by locality of the weak derivatives 
$\Ln\big(\{\nabla u_{x_0}=\underline{0}\}\setminus
\{\nabla^2 u_{x_0}=\underline{0}\}\big)=0$. Being 
$\Lambda_{u_{x_0}}\subseteq\{\nabla u_{x_0}=\underline{0}\}$, we conclude that $\Delta u_{x_0}=0$ $\Ln$-a.e. in 
$\Lambda_{u_{x_0}}$, and therefore in view of \eqref{e:PDE u f1} we infer
\begin{equation*}
 w \Delta w=(u_{x_0}-v)(\Delta u_{x_0} -1)=
 \begin{cases}
  (u_{x_0}-v)\,f_{x_0} & \textrm{$\Ln$-a.e. $\Omega\setminus\Lambda_{u_{x_0}}$}\cr
  v & \textrm{$\Ln$-a.e. $\Lambda_{u_{x_0}}$.}
\end{cases} 
\end{equation*}
Instead, estimating the second and third terms on the right hand 
side of \eqref{e:monneau1} thanks to \eqref{d:v} yields
\begin{equation*}
\begin{split}
  \frac{1}{r^{n+3}} & \int_{B_r}|\nabla w|^2\,dx - \frac{2}{r^{n+4}}\int_{\partial B_r} w^2\, d\mathcal{H}^{n-1} 
  = \frac{1}{r^{n+3}}\int_{B_r}\left(|\nabla u_{x_0}|^2
  +|\nabla v|^2\right)\,dx\\
  & -\frac{2}{r^{n+3}}\int_{B_r}\diw(u_{x_0}\nabla v)\,dx + \frac{2}{r^{n+3}}\int_{B_r}u_{x_0}\,dx- \frac{2}{r^{n+4}}\int_{\partial B_r}w^2\, d\mathcal{H}^{n-1}\\
  &\stackrel{\eqref{e:Phiv}}{=} \frac{1}{r}\big(\Phi_{u_{x_0}}(x_0,r) -\Phi_v(x_0,r)\big) 
      -\frac{2}{r^{n+4}}\int_{\partial B_r}u_{x_0}\,\left(\langle\nabla v,\frac{x}{r}\rangle-2v\right)\,dx \\
&\stackrel{\eqref{e:en BU sing Monneau}}{=}\frac{1}{r}
\left(\Phi_{u_{x_0}}(x_0,r) - \Phi_{u_{x_0}}(x_0,0^+)\right)\,.
\end{split}
\end{equation*}
Then, \eqref{e:monneau1} rewrites as 
\begin{equation*}
 \begin{split}
  \frac{d}{dr}\left(\frac{1}{r^{n+3}}\int_{\partial B_r} w^2\, d\mathcal{H}^{n-1}\right)
  = &\frac{2}{r}\left(\Phi_u(x_0,r) - \Phi_u(x_0,0^+)\right) \\
  &+ \frac{2}{r^{n+3}}\int_{B_r\setminus\Lambda_{u_{x_0}}} (u_{x_0}-v)f_{x_0}\,dx + \frac{2}{r^{n+3}}\int_{B_r\cap\Lambda_{u_{x_0}}}v\,dx.
 \end{split}
\end{equation*}
Inequality \eqref{e:Phi(r)-Phi(0)} in Theorem~\ref{t:Wp>n}, the growth of the solution $u$ from free boundary points in \eqref{e:u e grad u limitate}, 
the $2$-homogeneity and positivity of $v$ yield the conclusion
(cf. \eqref{e:uno}-\eqref{e:tre}): 
\begin{equation*}
 \frac{d}{dr}\left(\frac{1}{r^{n+3}}\int_{\partial B_r} w^2\, d\mathcal{H}^{n-1}\right)
 \geq -\frac{C}{r} \int_0^r \frac{\omega(t)}{t}\,dt - \frac{C}{r^{n+1}}\int_{B_r\setminus\Lambda_{u_{x_0}}}|f_{x_0}|\,dx 
 = -\frac{C}{r} \int_0^r \frac{\omega(t)}{t}\,dt
\end{equation*}
for some $C=C(n,p,\Lambda,c_0,K,\|f\|_{L^\infty},\|\mathbb{A}\|_{W^{1,p}})>0$.
\end{proof}

\section{Free boundary analysis}\label{s:applications}
Weiss' and Monneau's quasi-monotonicity formulas proved in the Section~\ref{s:q-mon formula} are important tools
to deduce regularity of free boundaries for classical obstacle problems for variational energies both in the quadratic and in the nonlinear setting (see \cite{FoGeSp15, Ger17, FoGerSp17, Monneau03, PSU, Weiss}).
In this section we improve \cite[Theorems~4.12 and 4.14]{FoGeSp15} in the quadratic case weakening the regularity of the coefficients of the relevant energies. This is possible thanks to the above 
mentioned new quasi-monotonicity formulas.

In the ensuing proof we will highlight only the substantial changes since the arguments are essentially those given in 
\cite{FoGeSp15, Ger17}. 
In particular, we remark again that in the quadratic case the main differences concern the quasi-monotonicity formulas 
established for the quantity $\Phi_u$ rather than for the natural candidate related to $\mathscr{E}$. 

We follow the approach by Weiss \cite{Weiss} and Monneau \cite{Monneau03} for the free boundary analysis in Theorem~\ref{t:linear}.
\begin{proof}[Proof of Theorem~\ref{t:linear}]
First recall that we may establish the conclusions for the function $u=w-\psi$ introduced in Section~\ref{s:prel}.
Given this, the only minor change to be done to the arguments in \cite[Section~4]{FoGeSp15} is related to the 
freezing of the energy where the regularity of the coefficients plays a substantial role. 
More precisely, in the current framework for all $v\in W^{1,2}(B_1)$ we have
 \[%\begin{equation}\label{e:E-Phi}
 \left|\int_{B_1}\big(\mathbb{A}(rx)\nabla v, \nabla v\rangle + 2f(rx)v\big)\,dx - \int_{B_1}\big(|\nabla v|^2 + 2v\big)dx\right|
  \leq (%\omega_\mathbb{A}(r)
  r^{1-\frac np}+\omega_f(r))
  \int_{B_1}\big(|\nabla v|^2 + 2v\big)dx.
 \]%\end{equation}
We then describe shortly the route to the conclusion. To begin with recall that the 
quasi-monotonicity formulas established in \cite[Section~3]{FoGeSp15} 
are to be substituted by those in Section~\ref{s:q-mon formula}.
Then the $2$-homogeneity of blow up limits in \cite[Proposition~4.2]{FoGeSp15} now follows from Theorem~\ref{t:Wp>n}.
The quadratic growth of solutions from free boundary points contained 
in \cite[Lemma~4.3]{FoGeSp15}, that implies non degeneracy of 
blow up limits, is contained in Lemma~\ref{l:quadratic}.
The classification of blow up limits is performed exactly as in \cite[Proposition~4.5]{FoGeSp15}.
The conclusions of \cite[Lemma~4.8]{FoGeSp15}, a result instrumental for the uniqueness of blow up limits at regular points, 
can be obtained with essentially no difference. The proofs of \cite[Propositions~4.10,~4.11, 
Theorems~4.12,~4.14]{FoGeSp15} remain unchanged. The theses then follow at once.
\end{proof}

%%%%%%%%%%%%%%%%%%%%%%%%%%%%%%%%%%%%%%%%%%%%%%%%%%%%%%%%%%%%%%%%%
%
%	BIBLIOGRAPHY
%
%%%%%%%%%%%%%%%%%%%%%%%%%%%%%%%%%%%%%%%%%%%%%%%%%%%%%%%%%%%%%%%%%

\bibliographystyle{plain}

\end{document}